\documentclass[10pt,a4paper]{amsart}

\usepackage{a4wide}

\usepackage[english]{babel}

\usepackage[pdftex]{graphicx}\pdfcompresslevel=9

\usepackage[all]{xy}
\usepackage{amsmath,amsthm,amssymb,enumerate}
\usepackage{latexsym}
\usepackage{amscd}

\usepackage[colorlinks=true]{hyperref}

\newtheorem{theorem}[subsection]{Theorem}
\newtheorem{thm}[subsection]{Theorem}
\newtheorem{defn}[subsection]{Definition}
\newtheorem{prop}[subsection]{Proposition}

\newtheorem{lemma}[subsection]{Lemma}

\newtheorem{remark}[subsection]{Remark}
\newtheorem{rem}[subsection]{Remark}

\theoremstyle{definition}
\newtheorem{example}[subsection]{Example}

\newcommand{\cat}{\mathcal}

\newcommand{\R}{\mathbb R}
\newcommand{\Q}{\mathbb Q}
\newcommand{\Z}{\mathbb Z}
\newcommand{\N}{\mathbb N}
\newcommand{\C}{\mathbb C}
\newcommand{\CP}{\mathbb P}
\newcommand{\T}{\mathbb T}

\newcommand{\ft}{{\mathfrak{t}}}

\newcommand{\ve}{\vec{e}}
\newcommand{\vd}{\vec{d}}
\newcommand{\vo}{\vec{0}}

\newcommand{\Ii}{{\cat I}}

\newcommand{\Oo}{{\cat O}}

\newcommand{\Ll}{{\cat L}}
\newcommand{\Xx}{{\cat X}}

\newcommand{\om}{{\omega}}
\newcommand{\de}{{\delta}}

\newcommand{\la}{{\lambda}}

\DeclareMathOperator{\Diff}{Diff}

\DeclareMathOperator{\Hess}{Hess}
\DeclareMathOperator{\Det}{Det}

\newcommand{\op}{{\overline{\partial}}}
\newcommand{\p}{{\partial}}
\newcommand{\cp}{{{\C\CP}\,\!}}

\begin{document}

\title[K\"ahler-Sasaki geometry of toric symplectic cones]
{K\"ahler-Sasaki geometry of toric symplectic cones in action-angle coordinates}
\author{Miguel Abreu}
\address{Centro de An\'{a}lise Matem\'{a}tica, Geometria e Sistemas
Din\^{a}micos, Departamento de Matem\'atica, Instituto Superior T\'ecnico, Av.
Rovisco Pais, 1049-001 Lisboa, Portugal}
\email{mabreu@math.ist.utl.pt}

\thanks{Partially supported by the Funda\c c\~ao para a Ci\^encia e a 
Tecnologia (FCT/Portugal)}

\date{\today}

\begin{abstract}
In the same way that a contact manifold determines and is determined by a 
symplectic cone, a Sasaki manifold determines and is determined by a suitable 
K\"ahler cone. K\"ahler-Sasaki geometry is the geometry of these cones. 

This paper presents a symplectic action-angle coordinates approach to toric 
K\"ahler geometry and how it was recently generalized, by Burns-Guillemin-Lerman 
and Martelli-Sparks-Yau, to toric K\"ahler-Sasaki geometry. It also describes, 
as an application, how this approach can be used to relate a recent new family 
of Sasaki-Einstein metrics constructed by Gauntlett-Martelli-Sparks-Waldram 
in 2004, to an old family of extremal K\"ahler metrics constructed by Calabi 
in 1982. 
\end{abstract}

\keywords{toric symplectic cones; action-angle coordinates; symplectic 
potentials; K\"ahler, Sasaki and Einstein metrics}

\maketitle

\section{Introduction}
\label{s:intro}

This paper presents a particular symplectic approach to understand
the work of Boyer-Galicki~\cite{BG}, Lerman~\cite{Le}, 
Gauntlett-Martelli-Sparks-Waldram~\cite{GW1, GW2}, 	  	  
Burns-Guillemin-Lerman~\cite{BGL2} and 
Martelli-Sparks-Yau~\cite{MSY}, 
regarding the following general geometric set-up:
\vspace*{0.5cm}
\begin{center}
\ 
\xymatrix{
\ar@{}[r]|<*\txt{\footnotesize{toric}\\\footnotesize{contact form}\\\footnotesize{Sasaki}\\\footnotesize{Einstein}} &
N^{2n+1} \ar[r]^-{\text{symplect.}} \ar[dd]|{\text{quotient}} & 
M^{2(n+1)}=N\times\R \ar[ddl]|{\text{reduction}} & 
& \ar@{}[l]|*\txt{\footnotesize{toric}\\\footnotesize{symplectic cone}\\\footnotesize{K\"ahler}\\\footnotesize{Ricci-flat}}\\
\\
\ar@{}[r]|<*\txt{\footnotesize{toric}\\\footnotesize{symplectic}\\\footnotesize{K\"ahler}\\\footnotesize{Einstein}} & B^{2n}= N/K  }
\end{center}
\vspace*{0.5cm}
The basic example is given by $B=\cp^n$ with the Fubini-Study metric,
$N = S^{2n+1}$ with the round metric and $M = \R^{2(n+1)}\setminus\{0\}$ 
with the flat Euclidean metric.

Let us start with a few comments on the top row of this diagram.
A \emph{contact} manifold determines, via \emph{symplectization}, and is 
determined, via $\R$-quotient, by a \emph{symplectic cone}. Hence, 
contact geometry can be thought of as the $\R$-invariant or $\R$-equivariant 
geometry of symplectic cones. Similarly, a \emph{Sasaki} manifold determines 
and is determined by a suitable \emph{K\"ahler cone}. Hence, Sasaki geometry 
can be thought of as the $\R$-invariant or $\R$-equivariant geometry of these 
cones and that is what we mean by \emph{K\"ahler-Sasaki geometry}.

Recall that the symplectization $M$ of a (co-oriented) contact manifold
$N$ is diffeomorphic to $N\times \R$, but not in a canonical way. The
choice of a contact form on $N$ gives rise to a choice of such a splitting
diffeomorphism. Since any Sasaki manifold comes equipped with a contact
form, any K\"ahler-Sasaki cone comes equipped with a splitting diffeomorphism.

In our symplectic approach, a suitable K\"ahler cone is a symplectic cone
equipped with what we will call a \emph{Sasaki complex structure}, 
i.e. a suitable compatible complex structure. Such a cone will be called a 
\emph{K\"ahler-Sasaki cone} and the corresponding K\"ahler metric will be 
called a \emph{K\"ahler-Sasaki metric}.

When a K\"ahler-Sasaki metric is \emph{Ricci-flat}, the associated Sasaki 
metric is \emph{Einstein} with positive scalar curvature. There is a lot 
of interest on \emph{Sasaki-Einstein metrics} due to their possible relation 
with superconformal field theory via the conjectural AdS/CFT correspondence. 
For example, the above mentioned work of Gauntlett-Martelli-Sparks-Waldram, 
a group of mathematical physicists, is motivated by this.

Regarding the left column of the above diagram, recall that a choice of a
\emph{contact form} on a contact manifold $N$ gives rise to a 
\emph{Reeb vector field} $K$. Denote also by $K$ the contact $\R$-action 
given by its flow. The \emph{quotient} $B:=N/K$, when suitably defined, is a 
symplectic singular space. When $N^{2n+1}$ is Sasaki (resp. Sasaki-Einstein 
with scalar curvature $= n(2n+1)$), the Reeb vector field $K$ generates an
isometric flow and the quotient $B^{2n}$ is K\"ahler (resp. 
\emph{K\"ahler-Einstein} with scalar curvature $= 2n(n+1)$).

As Boyer-Galicki point out in the Preface of their recent book~\cite{BG1}, 
Sasaki geometry of $N$ is then naturally ``sandwiched" between two K\"ahler
geometries:
\begin{itemize}
\item[(i)] the K\"ahler geometry of the associated symplectic cone $M$;
\item[(ii)] the K\"ahler geometry of the base symplectic quotient $B$.
\end{itemize}
As it turns out, there is a direct symplectic/K\"ahler way to go from (i) to
(ii): \emph{symplectic/K\"ahler reduction}. That is why the symplectic 
approach of this paper will mostly forget $N$ and use only the diagonal part 
of the above diagram, i.e. $M$, $B$ and the reduction arrow between the two.

The word \emph{toric} implies that $M$ and $B$ admit a combinatorial
characterization via the images of the moment maps for the corresponding
torus actions:
\begin{itemize}
\item[(i)] a \emph{polyhedral cone} $C\subset\R^{n+1}$ for the toric symplectic cone 
$M^{2(n+1)}$;
\item[(ii)] a \emph{convex polytope} $P\subset\R^n$ for the toric symplectic space
$B^{2n}$.
\end{itemize}
The \emph{symplectic reduction} relation between $M$ and $B$ corresponds to $C$ being
a cone over $P$.

The word toric also implies that, in suitable symplectic 
\emph{action-angle coordinates}, the relevant compatible complex structures
on $M$ and $B$ can be described via \emph{symplectic potentials}, i.e. appropriate 
real functions on $C$ and $P$. It follows from a theorem of 
Calderbank-David-Gauduchon~\cite{CDG} that the \emph{K\"ahler reduction} relation 
between $M$ and $B$ gives rise to a direct explicit relation between the corresponding
symplectic potentials on $C$ and $P$. As an application, we can use this to show
that a particular family of K\"ahler-Einstein spaces, contained in a more general 
family of local $U(n)$-invariant extremal K\"ahler metrics constructed by Calabi in 1982~\cite{C2}, gives rise to Ricci-flat K\"ahler-Sasaki metrics on certain toric 
symplectic cones.

More precisely, let $n$, $m$ and $k$ be integers such that
\[
n\geq 2\,,\quad k\geq 1 \quad\text{and}\quad 0\leq m < kn\,.
\]
Consider the cone $C(k,m)\subset\R^{n+1}$ with $n+2$ facets defined by 
the following normals:
\begin{align}
\nu_i & = \left(\ve_i, 1\right)\,,\ i=1, \ldots, n-1\,; \notag \\
\nu_{n} & = \left((m+1)\ve_n - \vd, 1\right)\,; \notag \\
\nu_- & = \left( k \ve_n, 1 \right)\,; \notag \\
\nu_+ & = \left(- \ve_n, 1\right)\,; \notag 
\end{align}
where 
\[
\text{$\ve_i\in\R^n\,,\ i=1,\ldots,n$, are the canonical basis vectors and}
\ \vd = \sum_{i=1}^n \ve_i \in \R^n\,.
\]
Each of these cones $C(k,m)\subset\R^{n+1}$ is good, in the sense of
Definition~\ref{def:gcone}, hence defines a toric symplectic cone $M_{k,m}^{2(n+1)}$. 
Because their defining normals lie on a fixed hyperplane in $\R^{n+1}$, the first 
Chern class of all these symplectic cones is zero. 
\begin{thm}\label{thm:main}
When
\begin{equation} \label{cond:kse}
\frac{(k-1)n}{2} < m < kn
\end{equation}
the toric symplectic cone $M_{k,m}^{2(n+1)}$ has a Ricci-flat K\"ahler-Sasaki metric.
The corresponding reduced toric K\"ahler-Einstein space belongs to Calabi's family.
\end{thm}

Let $N_{k,m}^{2n+1}$ denote the corresponding toric Sasaki-Einstein manifold.
Using a result of Lerman~\cite{Le2}, one can easily check that $N_{k,m}^{2n+1}$ 
is simply connected iff
\begin{equation} \label{cond:1conn}
{\rm gcd} (m+n, k+1) = 1 \,.
\end{equation}
When $n=2$ one can determine an explicit relation between $N^5_{k,m}$ and the 
simply connected toric Sasaki-Einstein $5$-manifolds $Y^{p,q}$, $0<q<p$, 
${\rm gcd} (q, p) = 1$, constructed by Gauntlett-Martelli-Sparks-Waldram~\cite{GW1}. 
In fact, as we will see, the associated $3$-dimensional moment cones are $SL(3,\Z)$ 
equivalent iff $k=p-1$ and $m=p+q-2$. Note that in this case 
\[
\frac{(k-1)n}{2} < m < k n \Leftrightarrow 0 < q < p
\]
and
\[
{\rm gcd} (m+n, k+1) = 1 \Leftrightarrow {\rm gcd} (q, p) = 1 \ .
\]
Since
\[
Y^{p,q} \cong S^2 \times S^3 \quad\text{for all $0<q<p$ such that  
${\rm gcd} (q, p) = 1$,}
\]
we conclude that
\[
N^5_{k,m} \cong S^2 \times S^3 \quad\text{for all $k,m\in\N$
satisfying~(\ref{cond:kse}) and~(\ref{cond:1conn}) (with $n=2$).}
\]

Gauntlett-Martelli-Sparks-Waldram construct in~\cite{GW2} higher dimensional 
generalizations of the manifolds $Y^{p,q}$. They do not describe their exact
diffeomorphism type and they do not write down the associated moment cones. 
The later should be $SL(n+1,\Z)$ equivalent to the cones $C(k,m)\subset\R^{n+1}$, 
with $k,m\in\N$ satisfying~(\ref{cond:kse}) and~(\ref{cond:1conn}), while the
former should be diffeomorphic to the corresponding 
$N_{k,m}^{2n+1}\subset M_{k,m}^{2(n+1)}$. The cones $C(k,m)\subset\R^{n+1}$
can be used to determine the diffeomorphism type of these manifolds. The 
following theorem is a particular example of that.

\begin{thm} \label{thm:main-2}
Given $n\geq 2$ and $m\in\N$, consider the (toric) complex manifold of 
real dimension $2n$ given by
\[
H^{2n}_m := \CP (\Oo(-m)\oplus\C) \to \C\CP^{n-1}\,.
\]
When $k=1$ and $0<m<n$, the toric symplectic cone $M_{1,m}^{2(n+1)}$ is
diffeomorphic to the total space of the anti-canonical line bundle of
$H^{2n}_m$ minus its zero section, while the toric contact manifold
$N_{1,m}^{2n+1}$ is diffeomorphic to the total space of the corresponding
circle bundle.
\end{thm}
\begin{rem} \label{rem:main-2}
Theorems~\ref{thm:main} and~\ref{thm:main-2} give rise to two natural
sub-actions of the torus action on the toric contact manifold 
$N_{1,m}^{2n+1}$:
\begin{itemize}
\item[(i)] the $\R$-action given by the flow of the Reeb vector field $K$,
determined by the contact form associated with the Sasaki-Einstein metric
given by Theorem~\ref{thm:main};
\item[(ii)] the $S^1$-action coming from the identification between
$N_{1,m}^{2n+1}$ and an $S^1$-bundle over $H^{2n}_m$.
\end{itemize}
Although in other more \emph{regular} examples, like the basic one given by
an odd-dimensional round sphere, the analogues of these two actions coincide, 
they cannot coincide in the present situation. If that were the case, we 
would have that $H^{2n}_m$ could be identified with $N_{1,m}^{2n+1}/K$ and
would then admit a K\"ahler-Einstein metric. That is well-known to be false.
In fact, the complex manifolds $H^{2n}_m$ are used by Calabi~\cite{C2} as 
examples that do not admit any K\"ahler-Einstein metric but do admit explicit 
extremal K\"ahler metrics. 

As we will see, the quotient $N_{1,m}^{2n+1}/K$ can be identified via its
moment polytope as a toric symplectic quasifold, in the sense of 
Prato~\cite{Pr}.
\end{rem}

The paper is organized as follows. In Section~\ref{s:toric} we give some
background on symplectic toric orbifolds and recall the definition and 
properties of symplectic potentials for toric compatible complex structures.
Section~\ref{s:tscones} is devoted to symplectic cones, their relation
with co-oriented contact manifolds and the classification of toric symplectic
cones via their moment polyhedral cones. The definition and basic properties
of (toric) K\"ahler-Sasaki cones is the subject of Section~\ref{s:tkscones},
which includes a brief description of their relation with (toric) Sasaki
manifolds. Cone action-angle coordinates and symplectic potentials are
introduced in Section~\ref{s:aacoord}, where we also discuss the behaviour
of symplectic potentials and toric K\"ahler-Sasaki metrics under symplectic
reduction. Section~\ref{s:newold-2} contains the proofs of 
Theorems~\ref{thm:main} and~\ref{thm:main-2}.

\subsection{Acknowledgments} I thank A.~Cannas da Silva and R.~Loja 
Fernandes, organizers of the Geometry Summer School, Instituto Superior 
T\'ecnico, Lisbon, Portugal, July 2009, where this work was presented as part 
of a mini-course. I also thank Gustavo Granja and Jos\'e Nat\'ario for useful 
conversations, and an anonymous referee for several comments and sugestions
that improved the exposition.

\section{Toric K\"ahler Orbifolds} 
\label{s:toric}

In this section, after some preliminary background on symplectic toric orbifolds, 
we recall the definition and some properties of symplectic potentials for compatible 
toric complex structures in action-angle coordinates, including a formula for the 
scalar curvature of the corresponding toric K\"ahler metric. For details 
see~\cite{Abr2,Abr3}.

\subsection{Preliminaries on Toric Symplectic Orbifolds}

\begin{defn} \label{def:torb}
A \emph{toric symplectic orbifold} is a connected $2n$-dimensional
symplectic orbifold $(B,\om)$ equipped with an effective Hamiltonian
action $\tau:\T^n \to \Diff (B,\om)$ of the standard (real) $n$-torus 
$\T^n = \R^n/2\pi\Z^n$. The corresponding \emph{moment map}, well-defined up 
to addition by a constant, will be denoted by $\mu : B \to \ft^\ast \cong \R^n$.
\end{defn}

When $B$ is a compact smooth manifold, the Atiyah-Guillemin-Sternberg 
convexity theorem states that the image $P=\mu(B)\subset \R^n$ of the 
moment map $\mu$ is the convex hull of the image of the points in $B$ 
fixed by $\T^n$, i.e. a convex polytope in $\R^n$. A theorem of 
Delzant~\cite{De} then says that the convex polytope $P\subset \R^n$ 
completely determines the toric symplectic manifold, up to equivariant 
symplectomorphisms.

In~\cite{LeTo} Lerman and Tolman generalize these two theorems to
orbifolds. While the convexity theorem generalizes word for word,
one needs more information than just the convex polytope $P$ to
generalize Delzant's classification theorem.

\begin{defn} \label{def:lapo}
A convex polytope $P$ in $\R^n$ is called \emph{simple} and
\emph{rational} if:
\begin{itemize}
\item[{\bf (1)}] there are $n$ edges meeting at each vertex $p$;
\item[{\bf (2)}] the edges meeting at the vertex $p$ are rational, 
i.e. each edge is of the form $p + tv_i,\ 0\leq t\leq \infty,\ 
{\rm where}\ v_i\in\Z^n$;
\item[{\bf (3)}] the $v_1, \ldots, v_n$ in (2) can be chosen to be a 
$\Q$-basis of the lattice $\Z^n$.
\end{itemize}
A \emph{facet} is a face of $P$ of codimension one. Following Lerman-Tolman,
we will say that a \emph{labeled polytope} is a rational simple convex
polytope $P\subset \R^n$, plus a positive integer (\emph{label}) attached
to each of its facets.

Two labeled polytopes are \emph{isomorphic} if one can be mapped to the other
by a translation, and the corresponding facets have the same integer labels.
\end{defn}

\begin{rem} 
In Delzant's classification theorem for compact symplectic
toric manifolds, there are no labels (or equivalently, all labels are
equal to $1$) and the polytopes that arise are
slightly more restrictive: the ``$\Q$'' in (3) is replaced by ``$\Z$''.
These are called \emph{Delzant polytopes}.
\end{rem}

\begin{rem} \label{rem:labels}
Each facet $F$ of a rational simple convex polytope $P\subset \R^n$
determines a unique lattice vector $\nu_F\in\Z^n \subset\R^n$: the primitive
inward pointing normal lattice vector. A convenient way of thinking about 
a positive integer label $m_F\in\N$ associated to $F$ is by dropping the primitive 
requirement from this lattice vector: consider $m_F\nu_F$ instead of $\nu_F$.

In other words, a labeled polytope can be defined as a rational simple polytope
$P\subset \R^n$ with an inward pointing normal lattice vector associated 
to each of its facets. When dealing with the effect of affine transformations
on labeled polytopes it will also be useful to allow more general inward
pointing normal vectors (see the end of this section).
\end{rem}

\begin{thm}[Lerman-Tolman] \label{thm:LeTo}
Let $(B,\om,\tau)$ be a compact toric symplectic orbifold, with
moment map $\mu : B \to \R^n$. Then $P\equiv \mu(B)$ is a
rational simple convex polytope. For every facet $F$ of $P$, there
exists a positive integer $m_F$, the label of $F$, such that the
structure group of every $p\in \mu^{-1}(\breve{F})$ is
$\Z/m_F\Z$ (here $\breve{F}$ is the relative interior of $F$).

Two compact toric symplectic orbifolds are equivariant symplectomorphic
(with respect to a fixed torus acting on both) if and only if their
associated labeled polytopes are isomorphic. Moreover, every labeled
polytope arises from some compact toric symplectic orbifold.
\end{thm}

Recall that a \emph{K\"ahler orbifold} can be defined as a symplectic
orbifold $(B,\omega)$ equipped with a \emph{compatible} complex 
structure $J\in\Ii(B,\omega)$, i.e. a complex structure on $B$ such 
that the bilinear form
\[
g_J (\cdot, \cdot) := \omega (\cdot, J\cdot)
\]
defines a \emph{Riemannian metric}. The proof of Theorem~\ref{thm:LeTo}, 
in both manifold and orbifold cases, gives an explicit construction of a 
canonical model for each toric symplectic orbifold, i.e. it associates to 
each labeled polytope $P$ an explicit toric symplectic orbifold 
$(B_P,\om_P,\tau_P)$ with moment map $\mu_P:B_P\to P$. Moreover, this
explicit construction consists of a certain symplectic reduction of 
the standard $\C^d$, for $d=$ number of facets of $P$, to which one can
apply the K\"ahler reduction theorem of Guillemin and Sternberg~\cite{GS}.
Hence, the standard complex structure on $\C^d$ induces a canonical 
$\T^n$-invariant complex structure $J_P$ on $B_P$, compatible with $\om_P$. 
In other words, each toric symplectic orbifold is K\"ahler and to each 
labeled polytope $P\subset \R^n$ one can associate a canonical toric 
K\"ahler orbifold $(B_P,\om_P,J_P,\tau_P)$ with moment map $\mu_P:B_P\to P$.

\subsection{Symplectic Potentials for Toric Compatible Complex Structures}

Toric compatible complex structures, and corresponding K\"ahler metrics, 
can be described using the following symplectic set-up.

Let $\breve{P}$ denote the interior of $P$, and
consider $\breve{B}_P\subset B_P$ defined by $\breve{B}_P = \mu_P^{-1}
(\breve{P})$. One can easily check that $\breve{B}_P$ is a smooth open 
dense subset of $B_P$, consisting of all the points where the 
$\T^n$-action is free. It can be described as
\[
\breve{B}_P\cong \breve{P}\times \T^n =
\left\{ (x,y): x\in\breve{P}\subset\R^n\,,\ 
y\in\R^n/2\pi\Z^n\right\}\,,
\]
where $(x,y)$ are symplectic or \emph{action-angle} coordinates for
$\om_P$, i.e.
\[
\om_P = dx\wedge dy = \sum_{j=1}^n dx_j \wedge dy_j\ .
\]

If $J$ is any $\om_P$-compatible toric complex structure on $B_P$,
the symplectic $(x,y)$-coordinates on $\breve{B}_P$ can be chosen
so that the matrix that represents $J$ in these coordinates has the form
\[
\begin{bmatrix}
\phantom{-}0\ \  & \vdots & -S^{-1} \\
\hdotsfor{3} \\
\phantom{-}S\ \  & \vdots & 0\,
\end{bmatrix}
\]
where $S=S(x)=\left[s_{jk}(x)\right]_{j,k=1}^{n,n}$ is a symmetric and
positive-definite real matrix. A simple computation shows that the vanishing 
of the Nijenhuis tensor, i.e. the integrability condition for the complex 
structure $J$, is equivalent to $S$ being the Hessian of a smooth function 
$s\in C^\infty (\breve{P})$, i.e.
\[
S = \Hess_x (s)\,,\ s_{jk}(x) = \frac{\p^2 s}{\p x_j \p x_k} (x)\,,\ 
1\leq j,k \leq n\,.
\]
Holomorphic coordinates for $J$ are given in this case by
\[
z(x,y) = u(x,y) + i v(x,y) = \frac{\p s}{\p x}(x)
+ iy\ .
\]
We will call $s$ the \emph{symplectic potential} of the 
compatible toric complex structure $J$. Note that the K\"ahler metric
$g_J (\cdot,\cdot) = \om_P(\cdot,J\cdot)$ is given in these
$(x,y)$-coordinates by the matrix
\begin{equation} \label{metricG}
\begin{bmatrix}
\phantom{-}S & \vdots & 0\  \\
\hdotsfor{3} \\
\phantom{-}0 & \vdots & S^{-1}
\end{bmatrix} \,.
\end{equation}

\begin{rem} \label{rmk:don1}
A beautiful proof of this local normal form for toric compatible complex
structures is given by Donaldson in~\cite{D4} (see also~\cite{Abr4}).
It illustrates a small part of his formal general framework for the action 
of the symplectomorphism group of a symplectic manifold on its space of 
compatible complex structures (cf.~\cite{D1}).
\end{rem}

We will now characterize the symplectic potentials that correspond 
to toric compatible complex structures on a toric symplectic orbifold 
$(B_P,\om_P,\tau_P)$. Every convex rational simple polytope 
$P\subset \R^n$ can be described by a set of inequalities of the form
\[
\langle x, \nu_r\rangle + \rho_r \geq 0\,,\ r=1,\ldots,d,
\]
where $d$ is the number of facets of $P$, each $\nu_r$ is a 
primitive element of the lattice $\Z^n\subset\R^n$ (the inward-pointing
normal to the $r$-th facet of P), and each $\rho_r$ is a real number.
Following Remark~\ref{rem:labels}, the labels $m_r\in\N$ attached to the 
facets can be incorporated in the description of $P$ by considering the 
affine functions $\ell_r : \R^n \to \R$ defined by
\[
\ell_r (x) = \langle x, m_r\nu_r \rangle + \la_r\,\ 
\mbox{where}\ \la_r = m_r\rho_r\ \mbox{and}\ 
r=1,\ldots,d\,.
\]
Then $x$ belongs to the $r$-th facet of $P$ iff $\ell_r (x) = 0$, 
and $x\in\breve{P}$ iff $\ell_r(x) > 0$ for all $r=1,\ldots,d$.

The following  two theorems are proved in~\cite{Abr3}. The first is a straightforward generalization to toric orbifolds of a result of Guillemin~\cite{Gui1}.
\begin{theorem} \label{thm1}
Let $(B_P,\om_P,\tau_P)$ be the symplectic toric orbifold associated
to a labeled polytope $P\subset \R^n$. Then, in suitable
action-angle $(x,y)$-coordinates on $\breve{B}_P\cong\breve{P}
\times \T^n$, the symplectic potential $s_P\in C^\infty (\breve{P})$ of the canonical compatible toric complex structure $J_P$ is given by
\[
s_P (x) = \frac{1}{2} \sum_{r=1}^{d} \ell_r(x) \log \ell_r (x)\ .
\]
\end{theorem}
The second theorem provides the symplectic version of the 
$\p\op$-lemma in this toric orbifold context.
\begin{theorem} \label{thm2}
Let $J$ be any compatible toric complex structure on the symplectic
toric orbifold $(B_P,\om_P,\tau_P)$. Then, in suitable
action-angle $(x,y)$-coordinates on $\breve{B}_P\cong\breve{P}
\times \T^n$, $J$ is given by a symplectic potential
$s\in C^\infty(\breve{P})$ of the form
\[
s(x) = s_P (x) + h(x)\,,
\]
where $s_P$ is given by Theorem~\ref{thm1}, $h$ is smooth on the whole
$P$, and the matrix $S=\Hess(s)$ is positive definite on $\breve{P}$
and has determinant of the form
\[
\Det(S) = \left(\de \prod_{r=1}^d \ell_r \right)^{-1}\,,
\]
with $\de$ being a smooth and strictly positive function on the whole $P$.

Conversely, any such potential $s$ determines a 
complex structure on $\breve{B}_P\cong\breve{P}\times \T^n$, that
extends uniquely to a well-defined compatible toric complex structure $J$
on the toric symplectic orbifold $(B_P,\om_P,\tau_P)$.
\end{theorem}

\subsection{Scalar Curvature}

We now recall from~\cite{Abr1} a particular formula for the scalar
curvature in action-angle $(x,y)$-coordinates. A K\"ahler metric of the
form~(\ref{metricG}) has scalar curvature $Sc$ given by
\[
Sc = - \sum_{j,k} \frac{\p}{\p x_j}
\left( s^{jk}\, \frac{\p \log \Det(S)}{\p x_k} \right)\,,
\]
which after some algebraic manipulations becomes the more compact
\begin{equation} \label{scalarsymp2}
Sc = - \sum_{j,k} \frac{\p^2 s^{jk}}{\p x_j \p x_k}\,, 
\end{equation}
where the $s^{jk},\ 1\leq j,k\leq n$, are the entries of the inverse 
of the matrix $S = \Hess_x (s)$, $s\equiv$ symplectic potential
(Donaldson gives in~\cite{D4} an appropriate interpretation of this 
formula, by viewing the scalar curvature as the moment map for the
action of the symplectomorphism group on the space of compatible 
complex structures).

\subsection{Symplectic Potentials and Affine Transformations}

The labeled polytope $P\subset\R^n$ of a symplectic toric
orbifold is only well defined up to translations, since the
moment map is only well defined up to addition of constants. 
Moreover, the twisting of the action by an automorphism of
the torus $\T^n = \R^n / 2\pi\Z^n$ corresponds to an $SL(n,\Z)$ 
transformation of the polytope. Since these operations have no
effect on a toric K\"ahler metric, symplectic potentials should
have a natural transformation property under these affine maps.
While the effect of translations is trivial to analyse, the effect 
of $SL(n,\Z)$ transformations is more interesting. In fact:
\[
\text{symplectic potentials transform quite naturally under any
$GL(n,\R)$ linear transformation.}
\]
Let $T\in GL(n,\R)$ and consider the linear symplectic change
of action-angle coordinates
\[
x := T^{-1}x'\quad\text{and}\quad y:=T^t y'\,.
\]
Then 
\[
P' = \bigcap_{a=1}^d \{x'\in\R^n\,:\ 
\ell'_a (x') := \langle x', \nu'_a \rangle + \lambda'_a \geq 0\} 
\]
becomes
\[
P:= T^{-1}(P') = \bigcap_{a=1}^d \{x\in\R^n\,:\ \ell_a (x) := 
\langle x, \nu_a \rangle + \lambda_a \geq 0\}
\]
with 
\[
\nu_a = T^t \nu'_a\quad\text{and}\quad \lambda_a = \lambda'_a\,,
\]
and symplectic potentials transform by
\[
s = s'\circ T \quad\text{(in particular, $s_{P} = s_{P'} \circ T)$.}
\]
The corresponding Hessians are related by
\[
S = T^t (S'\circ T) T
\]
and
\[
Sc = Sc' \circ T\,.
\]

For the purposes of this paper, the point of this discussion is the following.
Let $P\subset\R^n$ be a labeled polytope and $P' = T(P) \subset\R^n$ for some
arbitrary $T\in GL(n,\R)$. Supose that
\[
s' : \breve{P}'\to\R
\]
is of the form specified in Theorem~\ref{thm2} (with $s_{P'} = s_{P}\circ T^{-1}$).
Then
\[
s := s' \circ T : \breve{P} \to \R
\]
also has the form specified in Theorem~\ref{thm2} and, consequently, is the
symplectic potential of a well defined toric compatible complex structure on
the toric symplectic orbifold $(B_P, \omega_P)$. Moreover, since $Sc = Sc' \circ T$,
we have that
\[
Sc' = \text{constant} \Leftrightarrow Sc = \text{constant}\,.
\]

\begin{example} \label{ex:hirzebruch}
Figure~\ref{fig:hirzebruch} illustrates two equivalent descriptions 
of a toric symplectic rational ruled $4$-manifold or, equivalently,
of a Hirzebruch surface 
\[
H^2_m := \CP (\Oo(-m)\oplus\C) \to \C\CP^{1}\,,\  m\in\N\,. 
\]
The linear map $T\in GL(2,\R)$ relating the two is given by
\[
T = 
\begin{bmatrix}
m & -1 \\
0 & \ 1
\end{bmatrix}    
\]
The inward pointing normal that should be considered for each facet is
specified. The right polytope is a standard Delzant polytope for
the Hirzebruch surface $H^2_m$. The left polytope is very useful for
the constructions of section~\ref{s:newold-2} and
was implicitly used by Calabi in~\cite{C2}.
\end{example}

\begin{figure}
      \centering
      \includegraphics[viewport = 50 0 250 100, scale=1.0]{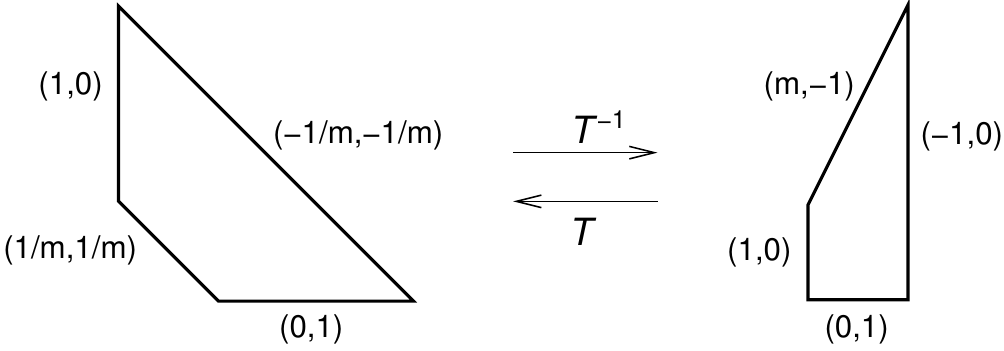}
      \caption{Hirzebruch surfaces.}
      \label{fig:hirzebruch}
\end{figure}

\section{Toric Symplectic Cones} 
\label{s:tscones}

In this section, after defining symplectic cones and briefly reviewing
their direct relation with co-oriented contact manifolds, we consider 
toric symplectic cones and their classification via good moment cones. 

\begin{defn} \label{def:scone}
A \emph{symplectic cone} is a triple $(M, \omega, X)$, where 
$(M,	\omega)$ is a connected symplectic manifold, i.e. 
$\omega\in\Omega^2(M)$ is a closed and non-degenerate $2$-form, 
and $X\in\Xx (M)$ is a vector field generating a proper 
$\R$-action $\rho_t:M\to M$, $t\in\R$, such that
$\rho_t^\ast (\omega) = e^{2t} \omega$. Note that the 
\emph{Liouville vector field} $X$ satisfies
$\Ll_X \omega = 2\omega$, or equivalently
\[
\omega = \frac{1}{2} d (\iota(X) \omega)\,.
\]
A \emph{compact} symplectic cone is a symplectic cone
$(M, \omega, X)$ for which the quotient $M/\R$ is compact.
\end{defn}

\begin{defn} \label{def:ccontact}
A \emph{co-orientable contact manifold} is a pair $(N, \xi)$, where
$N$ is a connected odd dimensional manifold and $\xi \subset TN$ 
is a maximally non-integrable hyperplane distribution globally defined 
by some \emph{contact form} $\alpha\in\Omega^1(N)$, i.e.
\[
\xi = \ker \alpha \quad\text{and}\quad
d\alpha|_\xi\ \text{is non-degenerate.}
\]
A \emph{co-oriented} contact manifold is a triple $(N, \xi, [\alpha])$,
where $(N,\xi)$ is a co-orientable contact manifold and $[\alpha]$ is the
conformal class of some contact form $\alpha$, i.e.
\[
[\alpha] = \left\{e^h \alpha \ |\ h\in C^\infty(N) \right\}\,.
\]
\end{defn}

Given a co-oriented contact manifold $(N, \xi, [\alpha])$, with contact
form $\alpha$, let
\[
M:=N\times\R\,,\ \omega := d(e^t \alpha) \quad\text{and}\quad
X:= 2\frac{\partial}{\partial t}\,,
\]
where $t$ is the $\R$ coordinate. Then $(M,\omega,X)$ is a 
symplectic cone, usually called the \emph{symplectization} of
$(N, \xi, [\alpha])$.

Conversely, given a symplectic cone $(M, \omega, X)$ let 
\[
N := M/\R\,,\ \xi := \pi_\ast (\ker (\iota(X) \omega))
\quad\text{and}\quad 
\alpha := s^\ast (\iota(X) \omega)\,,
\]
where $\pi:M\to N$ is the natural principal $\R$-bundle quotient 
projection and $s:N\to M$ is any global section (note that such global 
sections always exist, since any principal $\R$-bundle is trivial). 
Then $(N,\xi, [\alpha])$ is a  co-oriented contact manifold whose 
symplectization is the symplectic cone $(M, \omega, X)$.

In fact, we have that
\begin{center}
co-oriented contact manifolds $\overset{1:1}{\longleftrightarrow}$ 
symplectic cones
\end{center}
(see Chapter 2 of~\cite{Le1} for details).
Under this bijection, compact toric contact manifolds, Sasaki manifolds 
and Sasaki-Einstein metrics correspond respectively to the toric 
symplectic cones, K\"ahler-Sasaki cones and Ricci-flat K\"ahler-Sasaki 
metrics that are the subject of this paper.

\begin{example} \label{ex:R}
The most basic example of a symplectic cone is $\R^{2(n+1)}\setminus\{0\}$ with 
linear coordinates 
\[
(u_1, \ldots, u_{n+1}, v_1, \ldots, v_{n+1})\,,
\]
symplectic form  
\[
\om_{\rm st} = du \wedge dv := \sum_{j=1}^{n+1} du_j \wedge dv_j
\]
and Liouville vector field
\[
X_{\rm st} = u\frac{\p}{\p u} +  v\frac{\p}{\p v}
:= \sum_{j=1}^{n+1} \left( u_j\frac{\p}{\p u_j} +  v_j\frac{\p}{\p v_j}\right)\,.
\]
The associated co-oriented contact manifold is isomorphic to $(S^{2n+1}, \xi_{\rm st})$,
where $S^{2n+1} \subset \C^{n+1}$ is the unit sphere and $\xi_{\rm st}$ is the hyperplane 
distribution of complex tangencies, i.e.
\[
\xi_{\rm st} = T S^{2n+1} \cap i \, T S^{2n+1}\,.
\]
\end{example}

\begin{example} \label{ex:cotbundle}
Let $Q$ be a manifold and denote by $M$ the cotangent bundle of $Q$ with
the zero section deleted: $M:=T^\ast Q \setminus 0$. We have that $M$ is a 
symplectic cone since the proper $\R$-action $\rho_t:M\to M$, given by 
$\rho_t (q,p) = (q, e^{2t} p)$, expands the canonical symplectic
form exponentially. The associated co-oriented contact manifold is the
co-sphere bundle $S^\ast Q$.
\end{example}

\begin{example} \label{ex:boothby-wang}
Let $(B,\omega)$ be a symplectic manifold such that the cohomology class 
\[
\frac{1}{2\pi}[\omega] \in H^2(B,\R) \ 
\text{is integral, i.e. in the image of the natural map $H^2(B,\Z) \to H^2(B, \R)$.}
\] 
Suppose that $H^2(B,\Z)$ has no torsion, so that the above natural map is injective and
we can consider $H^2(B,\Z) \subset H^2 (B,\R)$. Denote by $\pi:N\to B$ the principle 
circle bundle with first Chern class 
\[
c_1 (N) = \frac{1}{2\pi}[\omega]\,.
\] 
A theorem of Boothby and Wang~\cite{BW} asserts that there is a connection $1$-form 
$\alpha$ on $N$ with $d\alpha = \pi^\ast\omega$ and, consequently, 
$\alpha$ is a contact form. We will call $(N, \xi:=\ker(\alpha))$ the \emph{Boothby-Wang} 
manifold of $(B, \omega)$. The associated symplectic cone is the total space of the 
corresponding line bundle $L\to B$ with the zero section deleted.

When $B=\cp^n$, with its standard Fubini-Study symplectic form, we recover 
Example~\ref{ex:R}, i.e. $(N,\xi) \cong (S^{2n+1}, \xi_{\rm st})$.
\end{example}

\begin{defn} \label{def:tscone}
A \emph{toric symplectic cone} is a symplectic cone $(M,\omega,X)$ of 
dimension $2(n+1)$ equipped with an effective $X$-preserving symplectic
$\T^{n+1}$-action, with moment map $\mu:M\to \ft^\ast \cong \R^{n+1}$ 
such that $\mu(\rho_t(m)) = e^{2t} \rho_t (m)\,,\ \forall\, m\in M,\, 
t\in\R$. Its \emph{moment cone} is defined to be the set
\[
C := \mu(M) \cup \{0\} \subset \R^{n+1}\,.
\]
\end{defn}

\begin{remark}
On a symplectic cone $(M,\omega,X)$, any $X$-preserving symplectic group 
action is Hamiltonian.
\end{remark}

\begin{example} \label{ex:toric-R}
Consider the usual identification $\R^{2(n+1)}\cong\C^{n+1}$ given by
\[
z_j = u_j + i v_j\,,\ j=1,\ldots, n+1\,,
\]
and the standard $\T^{n+1}$-action defined by
\[
(y_1, \ldots, y_{n+1}) \cdot (z_1,\ldots, z_{n+1}) = 
(e^{-i y_1}z_1, \ldots, e^{-i y_{n+1}}z_{n+1})\,.
\]
The symplectic cone $(\R^{2(n+1)}\setminus\{0\}, \om_{\rm st}, X_{\rm st})$
of Example~\ref{ex:R} equipped with this $\T^{n+1}$-action is a toric symplectic 
cone. The moment map $\mu_{\rm st}:\R^{2(n+1)}\setminus\{0\} \to \R^{n+1}$ is 
given by
\[
\mu_{\rm st} (u_1, \ldots, u_{n+1}, v_1, \ldots, v_{n+1}) = 
\frac{1}{2} (u_1^2 + v_1^2, \ldots, u_{n+1}^2 + v_{n+1}^2)\,.
\]
and the moment cone is 
$
C = (\R_0^+)^{n+1} \subset \R^{n+1}\,.
$
\end{example}

In~\cite{Le} Lerman completed the classification of compact toric 
symplectic cones, initiated by Banyaga and Molino~\cite{BnM1,BnM2,Bn} 
and continued by Boyer and Galicki~\cite{BG}. The ones that are relevant 
for toric K\"ahler-Sasaki geometry are characterized by having good 
moment cones.

\begin{defn}[Lerman] \label{def:gcone}
A cone $C\subset\R^{n+1}$ is \emph{good} if there exists a minimal set 
of primitive vectors $\nu_1, \ldots, \nu_d \in \Z^{n+1}$, with 
$d\geq n+1$, such that
\begin{itemize}
\item[(i)] $C = \bigcap_{a=1}^d \{x\in\R^{n+1}\,:\ 
\ell_a (x) := \langle x, \nu_a \rangle \geq 0\}$.
\item[(ii)] any codimension-$k$ face $F$ of $C$, $1\leq k\leq n$, 
is the intersection of exactly $k$ facets whose 
set of normals can be completed to an integral base of 
$\Z^{n+1}$.
\end{itemize}
\end{defn}

\begin{theorem}[Banyaga-Molino, Boyer-Galicki, Lerman] \label{thm:gcone}
For each good cone $C\subset\R^{n+1}$ there exists a unique compact toric 
symplectic cone $(M_C, \om_C, X_C, \mu_C)$ with moment cone $C$.
\end{theorem}
\begin{remark} \label{rmk:model}
The compact toric symplectic cones characterized by this theorem will
be called \emph{good} toric symplectic cones. Like for compact toric symplectic manifolds, the existence part of the theorem follows from an
explicit symplectic reduction construction starting from a symplectic 
vector space (see~\cite{Le}).
\end{remark}

\begin{example} \label{ex:stdcone}

Let $P\subset\R^n$ be an \emph{integral Delzant polytope}, 
i.e. a Delzant polytope with integral vertices or, equivalently,
the moment polytope of a compact toric symplectic manifold 
$(B_P, \om_P, \mu_P)$ such that $\frac{1}{2\pi}[\omega]\in H^2(B_P, \Z)$.
Then, its \emph{standard cone}
\begin{equation} \label{eq:stdcone}
C:= \left\{z(x,1)\in\R^{n}\times\R\,:\ x\in P\,,\ z\geq 0\right\}
\subset\R^{n+1}
\end{equation}
is a good cone. Moreover
\begin{itemize}
\item[(i)] the toric symplectic manifold $(B_P, \om_P, \mu_P)$ is the 
$S^1\cong \{{\bf 1}\}\times S^1 \subset \T^{n+1}$ \emph{symplectic reduction} 
of the toric symplectic cone $(M_{C}, \omega_{C}, X_{C}, \mu_{C})$ 
(at level one).
\item[(ii)] $(N_C:=\mu_C^{-1}(\R^{n}\times\{1\}), 
\alpha_C := (\iota(X_C) \omega_C)|_{N_C})$ is the \emph{Boothby-Wang} 
manifold of $(B_P, \om_P)$. The restricted $\T^{n+1}$-action makes it a 
\emph{toric contact manifold}.
\item[(iii)] $(M_C, \omega_C, X_C)$ is the \emph{symplectization} of 
$(N_C, \alpha_C)$.
\end{itemize}
See Lemma 3.7 in~\cite{Le3} for a proof of these facts.

If $P\subset\R^n$ is the standard simplex, i.e. $B_P = \cp^n$, 
then its standard cone $C\subset\R^{n+1}$ is the moment cone of
$(M_C = \C^{n+1}\setminus\{0\}, \om_{\rm st}, X_{\rm st})$
equipped with the $\T^{n+1}$-action given by
\begin{align}
& (y_1,\ldots,y_n,y_{n+1})\cdot (z_1, \ldots, z_n, z_{n+1}) \notag \\
=\ & (e^{-i(y_1+y_{n+1})}z_1, \ldots, e^{-i(y_n+y_{n+1})}z_n,
e^{-iy_{n+1}}z_{n+1})\notag
\end{align}
The moment map $\mu_C : \C^{n+1}\setminus\{0\} \to \R^{n+1}$ is
given by
\[
\mu_{C} (z) = 
\frac{1}{2} (|z_1|^2, \ldots, |z_n|^2, |z_1|^2+\cdots+|z_n|^2
+ |z_{n+1}|^2)
\]
and 
\[
N_C := \mu_C^{-1}(\R^{n}\times\{1\}) = \left\{ z\in\C^{n+1}\,:\ 
\|z\|^2 = 2 \right\} \cong S^{2n+1}\,.
\]
\end{example}

\begin{remark} \label{rem:orbBW}
Up to a possible twist of the action by an automorphism of the torus
$\T^{n+1}$, any good toric symplectic cone can be obtained via an 
orbifold version of the Boothby-Wang construction of 
Example~\ref{ex:boothby-wang}, where the base is a toric symplectic
orbifold. In fact, up to an $SL(n+1,\Z)$ transformation, any good moment 
cone can be written as the standard cone, given by~(\ref{eq:stdcone}), 
of a labeled polytope.
\end{remark}

\section{Toric K\"ahler-Sasaki Cones} 
\label{s:tkscones}

In this section we define (toric) K\"ahler-Sasaki cones, present their
basic properties and briefly describe their relation with (toric) 
Sasaki manifolds.

\begin{defn} \label{def:kscone}    
A \emph{K\"ahler-Sasaki cone} is a symplectic 
cone $(M,\omega,X)$ equipped with a compatible complex 
structure $J\in\Ii(M,\omega)$ such that the \emph{Reeb vector 
field} $K:=JX$ is K\"ahler, i.e.
\[
\Ll_K \omega = 0\quad\text{and}\quad\Ll_K J = 0\,.
\]
Note that $K$ is then also a \emph{Killing} vector field for the
Riemannian metric $g_J$.

Any such $J$ will be  called a \emph{Sasaki} complex structure
on the symplectic cone $(M,\omega,X)$ and the associated metric
$g_J$ will be called a \emph{K\"ahler-Sasaki} metric. 
The space of all Sasaki complex structures will be denoted by 
$\Ii_S (M,\omega, X)$.
\end{defn}

Given a K\"ahler-Sasaki cone $(M,\omega,X,J)$, define a smooth 
positive function $r:M\to\R^+$ by
\[
r:= \|X\| = \|JX\| = \|K\|\,,
\]
where $\|\cdot\|$ denotes the norm associated with the metric 
$g_J$. One easily checks that
\begin{itemize}
\item[(i)] $K$ is the \emph{Hamiltonian} vector field of $-r^2/2$;
\item[(ii)] $X$ is the \emph{gradient} vector field of $r^2/2$.
\end{itemize}
Define $\alpha \in \Omega^1 (M)$ by
\[
\alpha:= \iota(X)\omega / r^2\,.
\]
We then have that
\[
\omega = d (r^2 \alpha)/2\ ,\quad
\alpha (K) \equiv 1 \quad\text{and}\quad
\Ll_X \alpha = 0\ .
\]
If we now define
\[
N := \left\{ r = 1 \right\} \subset M
\quad\text{and}\quad 
\xi:=\ker{\alpha|_N}\,,
\] 
we have that
\[
(N, \xi, \alpha|_N, g_J|_N)\text{ is a \emph{Sasaki manifold}
(see~\cite{BG1} for the definition of a Sasaki manifold).}
\]
In fact, one can easily check from the definitions that
\begin{center}
Sasaki manifolds $\overset{1:1}{\longleftrightarrow}$ 
K\"ahler-Sasaki cones.
\end{center}

Given a K\"ahler-Sasaki cone $(M, \om, X, J)$, let
\[
B:=M//K = N/K
\] 
be the symplectic reduction of $(M, \om)$ by the action 
of $K=JX$ and denote by $\pi : N \to B$ the quotient projection. 
When $B$ is smooth, we have that $\pi^\ast (TB) \cong \xi$ and
$J|_\xi$ induces an almost complex structure on $B$ which, by the
already mentioned K\"ahler reduction theorem of Guillemin and 
Sternberg~\cite{GS}, is integrable. Hence,
\[
(B, d\alpha|_\xi, J|_\xi)\text{ is a K\"ahler manifold.}
\]
The smoothness of $B$ is related with the regularity of the K\"ahler-Sasaki cone.
\begin{defn} \label{def:regularKScone}
A K\"ahler-Sasaki cone $(M, \om, X, J)$, with Reeb vector field 
$K=JX$, is said to be:
\begin{itemize}
\item[(i)] \emph{regular} if $K$ generates a \emph{free} $S^1$-action.
\item[(ii)] \emph{quasi-regular} if $K$ generates a \emph{locally free} 
$S^1$-action.
\item[(iii)] \emph{irregular} if $K$ generates an \emph{effective} 
$\R$-action.
\end{itemize}
\end{defn}

\noindent Hence, $B$ is
\begin{itemize}
\item[(i)] a \emph{smooth} K\"ahler manifold if the K\"ahler-Sasaki cone is \emph{regular}.
\item[(ii)] a K\"ahler \emph{orbifold} if the K\"ahler-Sasaki cone is \emph{quasi-regular}.
\item[(iii)] only a K\"ahler \emph{quasifold}, in the sense of Prato~\cite{Pr},
if the K\"ahler-Sasaki cone is \emph{irregular}.
\end{itemize}

\begin{remark}
Note that the Sasaki manifold determined, as above, by a K\"ahler-Sasaki cone is 
always smooth.
\end{remark}

\begin{defn} \label{def:tKScone}    
A \emph{toric K\"ahler-Sasaki cone} is a good toric symplectic cone 
$(M,\omega,X, \mu)$ equipped with a \emph{toric} Sasaki complex 
structure $J\in\Ii_S^{\T}(M,\omega)$, i.e. a Sasaki complex 
structure invariant under the torus action. The associated metric $g_J$
will be called a toric K\"ahler-Sasaki metric.
\end{defn}
\begin{remark}
\ 
\begin{itemize}
\item[(i)] It follows from Theorem~\ref{thm:gcone} and Remark~\ref{rmk:model} that 
any good toric symplectic cone has toric Sasaki complex structures. These will be
described in the next section.
\item[(ii)] On a toric K\"ahler-Sasaki cone $(M,\omega,X, \mu, J)$, the 
K\"ahler action generated by the Reeb vector field $K=JX$ corresponds to the
action generated by a fixed vector in the Lie algebra of the torus
(see Lemma~\ref{lem:reebconst} below).
\item[(iii)] The K\"ahler reduction $B:=M//K$ of a toric K\"ahler-Sasaki cone is a 
toric K\"ahler space: manifold (regular case), orbifold 
(quasi-regular case) or quasifold (irregular case).
\end{itemize}

\end{remark}

\begin{example} \label{ex:toric-KS-R}
The toric symplectic cone 
$(\R^{2(n+1)}\setminus\{0\}, \om_{\rm st}, X_{\rm st}, \mu_{\rm st})$
of Example~\ref{ex:toric-R}, equipped with the standard linear complex
structure $J_0 : \R^{2(n+1)} 	\to \R^{2(n+1)}$ given by
\[
J_0 = 
\begin{bmatrix}
\phantom{-}0\ \  & \vdots & -I \\
\hdotsfor{3} \\
\phantom{-}I\ \  & \vdots & 0\,
\end{bmatrix}
\]
is a toric K\"ahler-Sasaki cone.
\end{example}

\section{Cone Action-Angle Coordinates and Symplectic Potentials} 
\label{s:aacoord}

As described in section~\ref{s:toric}, the space $\Ii^\T$ of toric compatible 
complex structures on a compact toric symplectic orbifold can be 
effectively parametrized, using global action-angle coordinates, by
symplectic potentials, i.e. certain smooth real valued functions on the 
corresponding labeled polytope. In this section we present the analogue of 
this fact for the space $\Ii_S^\T$ of toric Sasaki complex structures on a 
good toric symplectic cone, due to Burns-Guillemin-Lerman~\cite{BGL2} and Martelli-Sparks-Yau~\cite{MSY}. We will also discuss how symplectic potentials
and toric K\"ahler-Sasaki metrics behave under symplectic reduction.

Let $C \subset \R^{n+1}$ be a good cone and  $(M, \om, X, \mu)$ the 
corresponding good toric symplectic cone (we omit the subscript $C$ to 
simplify the notation). Let $\breve{C}$ 
denote the interior of $C$, and consider $\breve{M}\subset M$ defined by 
$\breve{M} = \mu^{-1}(\breve{C})$. One can easily check that $\breve{M}$ 
is a smooth open dense subset of $M$, consisting of all the points where 
the $\T^n$-action is free. One can use the explicit model for 
$(M, \om, X, \mu)$, given by the symplectic reduction construction mentioned
in Remark~\ref{rmk:model}, to show that $\breve{M}$ can be described as
\[
\breve{M}\cong \breve{C}\times \T^n =
\left\{ (x,y):\, x\in\breve{C},\, y\in\T^{n+1} \equiv 
\R^{n+1}/2\pi\Z^{n+1}\right\}\,,
\]
where in these $(x,y)$ coordinates we have
\[
\om|_{\breve{M}} = dx\wedge dy\  \quad
\mu(x,y) = x \quad\text{and}\quad
X|_{\breve{M}} = 2x \frac{\partial}{\partial x} = 2\sum_{i=1}^{n+1} x_i 
\frac{\partial}{\partial x_i}\,.
\]

\begin{defn} \label{def:aacoord}
Any such set of coordinates will be called \emph{cone action-angle} 
coordinates.
\end{defn}

If $J$ is any $\om$-compatible toric complex structure on $M$ such that
$\Ll_X J = 0$, i.e. for which the Liouville vector field $X$ is 
holomorphic, the cone action-angle $(x,y)$-coordinates on $\breve{M}$ can 
be chosen so that the matrix that represents $J$ in these coordinates has 
the form
\begin{equation}\label{matrixJ}
\begin{bmatrix}
\phantom{-}0\ \  & \vdots & -S^{-1} \\
\hdotsfor{3} \\
\phantom{-}S\ \  & \vdots & 0\,
\end{bmatrix}
\end{equation}
where $S=S(x)=\left[s_{ij}(x)\right]_{i,j=1}^{n+1,n+1}$ is a symmetric 
and positive-definite real matrix. The integrability condition for the
complex structure $J$ is again equivalent to $S$ being the Hessian of a 
smooth real function $s\in C^\infty (\breve{C})$, i.e.
\begin{equation} \label{Hessg}
S = \Hess_x (s)\,,\ s_{ij}(x) = \frac{\p^2 s}{\p x_i \p x_j} (x)\,,\ 
1\leq i,j \leq n+1\,,
\end{equation}
and holomorphic coordinates for $J$ are again given by
\[
z(x,y) = u(x,y) + i v(x,y) = \frac{{\partial}s}{{\partial}x}(x)
+ i y\,.
\]
The condition $\Ll_X J = 0$ is equivalent to 
\begin{equation} \label{hom-1}
S(e^{2t} x) = e^{-2t} S(x)\,,\ \forall \,t\in\R\,,\ x\in\breve{C}\,,
\end{equation}
i.e. to $S$ being homogeneous of degree $-1$ in $x$.

\begin{rem} \label{rmk:don2}
A proof of these facts can be given by combining Donaldson's method
of proof in the polytope case (cf. Remark~\ref{rmk:don1}) with the
Sasaki condition on the complex structure $J$.
\end{rem}

The Reeb vector field $K:=JX$ of such a toric complex structure 
(cf. Definition~\ref{def:kscone}) is given 
by
\[
K = \sum_{i=1}^{n+1} b_i \frac{\partial}{\partial y_i}
\quad\text{with}\quad
b_i = 2 \sum_{j=1}^{n+1} s_{ij} x_j\,.
\]
\begin{lemma}[Martelli-Sparks-Yau] \label{lem:reebconst}
If $S(x) = \left[s_{ij}(x)\right]$ is homogeneous of 
degree $-1$, then the corresponding Reeb vector field
$K= ({\bf 0}, K_s)$, with $K_s:=(b_1, \ldots, b_{n+1})$, is a 
\emph{constant} vector. In other words, the action 
generated by $K$ corresponds to the action generated by a fixed
vector in the Lie algebra of the torus. In particular, $K$ is K\"ahler 
and
\[
\text{regularity of the toric K\"ahler-Sasaki cone} 
\Leftrightarrow
\text{rationality of $K_s\in\R^{n+1}$.}
\]
\end{lemma}
The norm of the Reeb vector field is given by
\[
\|K\|^2 = \|({\bf 0}, K_s)\|^2 = b_i s^{ij} b_j =
b_i s^{ij} (2 s_{jk} x_k) = 2 b_i x_i = 
2 \langle x, K_s \rangle\,.
\]
Hence
\[
\|K\|>0 \Leftrightarrow \langle x, K_s \rangle > 0
\quad\text{and}\quad
\|K\|=1 \Leftrightarrow \langle x, K_s \rangle = 1/2\,.
\]
\begin{defn}[Martelli-Sparks-Yau] \label{def:char}
The \emph{characteristic hyperplane} $H_K$ and \emph{polytope} $P_K$ 
of a toric K\"ahler-Sasaki cone $(M,\omega,X, \mu, J)$, with moment 
cone $C\subset\R^{n+1}$, are defined as
\[
H_K := \{x\in\R^{n+1}\,:\ \langle x, K_s \rangle = 1/2\}
\quad\text{and}\quad
P_K := H_K \cap C\,.
\]
\end{defn}
\begin{rem} \label{rem:char}
Note that $N:=\mu^{-1}(H_K)$ is a toric Sasaki manifold 
and $P_K$ is the moment polytope of $B=M//K$. Moreover, we see that
$K$ gives rise to compatible splitting identifications 
$M = N\times \R$ and $C = P_K \times \R$
\end{rem}

As we have just seen, any toric Sasaki complex structure 
$J\in\Ii^{\T}_S (\breve{M}, \omega, X)$ can be written in 
suitable cone action-angle coordinates $(x,y)$ on $\breve{M} \cong \breve{C}\times\T^{n+1}$ in the form~(\ref{matrixJ}), with $S$ 
satisfying~(\ref{Hessg}) and~(\ref{hom-1}).
\begin{defn} \label{def:spotential}
The corresponding smooth real function $s\in C^\infty (\breve{C})$
will be called the \emph{symplectic potential} of the toric Sasaki 
complex structure
\end{defn}

\begin{example} \label{ex:sympot-R}
Consider the toric K\"ahler-Sasaki cone of Example~\ref{ex:toric-KS-R}.
In cone action-angle coordinates $(x,y)$ on
\[
\breve{C}\times\T^{n+1} = (\R^+)^{n+1} \times \T^{n+1}\,,
\]
the symplectic potential
\[
s: \breve{C} = (\R^+)^{n+1} \to \R
\]
of the toric Sasaki complex structure $J_0$ is given by
\[
s(x) = \frac{1}{2} \sum_{a=1}^{n+1} x_a \log x_a\,.
\]
\end{example}

We will now characterize the space of smooth real functions 
$s\in C^\infty (\breve{C})$ that are the symplectic potential
of some toric Sasaki complex structure 
$J\in\Ii^{\T}_S (M, \omega, X)$.

The K\"ahler reduction theorem of Guillemin and Sternberg can
also be applied to the symplectic reduction construction mentioned in 
Remark~\ref{rmk:model}. Hence, given a good cone $C\subset\R^{n+1}$, 
defined by
\[
C = \bigcap_{a=1}^d \{x\in\R^{n+1}\,:\ \ell_a (x) := \langle x, \nu_a 
\rangle \geq 0\}
\]
as in Definition~\ref{def:gcone}, the explicit model for the 
corresponding good toric symplectic cone $(M, \om, X, \mu)$ has a 
canonical toric Sasaki complex structure $J_C\in\Ii^{\T}_S (M, \omega, X)$. 
Its symplectic potential is given by the following particular case of a 
theorem proved by Burns-Guillemin-Lerman in~\cite{BGL2}.
\begin{thm}
\label{thm:cspot}
In appropriate action-angle coordinates $(x,y)$, the \emph{canonical symplectic potential} $s_C:\breve{C}\to\R$ for $J_C |_{\breve{C}}$ 
is given by
\[
s_C (x) = \frac{1}{2} \sum_{a=1}^d \ell_a (x) \log \ell_a (x)\,.
\]
\end{thm}
One checks easily that $\Hess_x (s_C)$ is homogeneous of degree 
$-1$. The corresponding Reeb vector field $K = ({\bf 0}, K_C)$ is given 
by
\begin{equation} \label{eq:K_C}
K_C = \sum_{a=1}^d \nu_a\,.
\end{equation}

\begin{example} \label{ex:sympot-R-2}
The symplectic potential presented in Example~\ref{ex:sympot-R} is the 
canonical symplectic potential of the corresponding good cone
$C = (\R^+_0)^{n+1} \subset \R^{n+1}$ and
\[
K_C = (1, \ldots, 1) \in \R^{n+1}\,.
\]
\end{example}

\begin{example} \label{ex:sympot-stdcone}

The standard cone over the standard simplex, considered in 
Example~\ref{ex:stdcone}, is given by
\[
C = \bigcap_{a=1}^{n+1} \{x\in\R^{n+1}\,:\ 
\ell_a (x) := \langle x, \nu_a \rangle \geq 0\}\,,
\]
where
\[
\nu_a = e_a\,,\ a=1,\ldots,n,\quad\text{and}\quad
\nu_{n+1} = (-1, \ldots , -1,1)\,.
\]

Hence, defining 
\[
r=\sum_{a=1}^n x_a\,,
\] 
we have that
\[
s_C(x) = \frac{1}{2} \left( \sum_{a=1}^n x_a \log x_a +
(x_{n+1}-r) \log (x_{n+1}-r) \right)
\]
and
\[
K_C = \sum_{a=1}^{n+1} \nu_a = (0, \ldots, 0, 1) \in \R^{n+1}\,.
\]

\end{example}

\begin{rem}
Examples~\ref{ex:sympot-R-2} and~\ref{ex:sympot-stdcone}
are isomorphic to each other under a $SL(n+1,\Z)$ transformation.
\end{rem}

Let $s,s':\breve{C}\to\R$ be two symplectic potentials defined on 
the interior of a cone $C\subset\R^{n+1}$. Then
\[ 
K_s = K_{s'} \Leftrightarrow (s-s') +\,\text{const. is 
homogeneous of degree $1$.}
\]

Given $b\in\R^{n+1}$, define
\begin{equation} \label{eq:s_b}
s_b (x) := \frac{1}{2} \left(\langle x,b \rangle 
\log \langle x,b \rangle - \langle x,K_C \rangle
\log\langle x,K_C \rangle \right)\,,
\end{equation}
with $K_C$ given by~(\ref{eq:K_C}).
Then $s:= s_C + s_b$ is such that $K_s = b$. 
If $C$ is good, this symplectic potential $s$ defines a 
\emph{smooth Sasaki} complex structure on the corresponding good
toric symplectic cone $(M, \omega, X, \mu)$ iff 
\[
\langle x,b \rangle > 0\,,\ \forall \, x\in C\setminus\{0\}\,, \ 
\text{i.e. $b\in\breve{C}^\ast$}
\] 
where $C^\ast\subset\R^{n+1}$ is the \emph{dual cone}
\[
C^\ast := \{x\in\R^{n+1}\,:\ \langle v, x \rangle \geq 0\,,\ 
\forall\, v\in C\}\,.
\]
This dual cone can be equivalently defined as
\[
C^\ast = \cap_\alpha \{x\in\R^{n+1}\,:
\ \langle \eta_\alpha, x \rangle \geq 0\}\,,
\]
where $\eta_\alpha \in \Z^{n+1}$ are the primitive generating edges 
of $C$.

\begin{thm}[Martelli-Sparks-Yau~\cite{MSY}]
Any toric Sasaki complex structure $J\in\Ii^\T_S$ on a good toric 
symplectic cone $(M, \omega, X, \mu)$, associated to a 
good moment cone $C\in\R^{n+1}$, is given by a symplectic potential
$s:\breve{C}\to\R$ of the form
\[
s = s_C + s_b + h\,, 
\]
where $s_C$ is the canonical potential, $s_b$ is given by~(\ref{eq:s_b}) 
with $b\in\breve{C}^\ast$, and $h:C\to\R$ is homogeneous of degree $1$ and 
smooth on $C\setminus\{0\}$.
\end{thm}

\subsection{Symplectic Reduction of Symplectic Potentials}

\begin{prop}[Calderbank-David-Gauduchon~\cite{CDG}] \label{prop:spotred}
Symplectic potentials \emph{restrict} naturally under toric symplectic reduction.
\end{prop}
More precisely, suppose $(M_P,\om_P,\mu_P)$ is a toric symplectic reduction 
of $(M_C, \om_C, \mu_C)$. Then there is an affine inclusion $P\subset C$ 
and 
\[
\text{any $\widetilde{J}\in\Ii^\T(M_C,\om_C)$ induces a reduced $J\in\Ii^{\T}(M_P,\om_P)$.}
\]
This proposition says that if 
\[
\text{$\tilde{s}:\breve{C}\to \R$ is a symplectic potential for 
$\widetilde{J}$}
\]
then
\[
\text{$s:=\tilde{s}|_{\breve{P}}:\breve{P}\to \R$ is a symplectic potential 
for $J$}.
\]
This property can be used to prove Theorems~\ref{thm1} and~\ref{thm:cspot}.
It is also particularly relevant for the following class of symplectic
potentials.
\begin{defn} \label{def:BWspot}
Let $P\subset\R^n$ be a convex polytope and 
$C\subset\R^{n+1}$ its standard cone given by~(\ref{eq:stdcone}). 
Given a symplectic potential $s:\breve{P}\to\R$, define its
\emph{Boothby-Wang} symplectic potential 
$\tilde{s}:\breve{C}\to\R$ by
\begin{equation} \label{eq:BW-pot}
\tilde{s}(x,z) := z \, s(x/z) + \frac{1}{2} z\log z
\,,\ \forall\, x\in\breve{P}\,,\ z\in\R^+\,.
\end{equation}
Note that
\[
K_{\tilde{s}} = (0, \ldots, 0, 1) \in \R^{n+1}\,.
\]
\end{defn}

\begin{example}
In general, 
\[
\tilde{s_P} \ne s_C\,.
\]
If $P = \bigcap_{a=1}^d \{x\in\R^n\,:\ \ell_a (x) := 
\langle x, \nu_a \rangle + \lambda_a \geq 0\}$,  consider
\[
s(x) = s_P (x) -\frac{1}{2} \ell_\infty (x) \log \ell_\infty (x)\,,
\]
where $\ell_\infty (x) := \sum_a \ell_a(x) = 
\langle x,\nu_\infty \rangle + \lambda_\infty$. Then
\[
\tilde{s}(x,z) = s_C (x,z) + s_b (x,z) \quad
\text{where $s_b$ is given by~(\ref{eq:s_b}) with $b= (0,\ldots,0,1)$.}
\]
\end{example}

\subsection{Toric K\"ahler-Sasaki-Einstein Metrics}

\begin{prop}
Let $P\subset\R^n$ be a convex polytope and $C\subset\R^{n+1}$ 
its standard cone defined by~(\ref{eq:stdcone}). Given a symplectic 
potential $s:\breve{P}\to\R$, let $\tilde{s}:\breve{C}\to\R$ be its 
Boothby-Wang symplectic potential given by~(\ref{eq:BW-pot}). Then
\[
\widetilde{Sc}(x,z) = \frac{Sc(x/z) - 2n (n+1)}{z}\,.
\]
In particular
\[
\widetilde{Sc} \equiv 0 \Leftrightarrow Sc \equiv 2n(n+1)
\]
and, when this happens, the corresponding toric Sasaki metric has
constant positive scalar curvature $= n(2n+1)$.
Moreover, 
\[
\text{$s$ defines a toric K\"ahler-Einstein metric with 
$Sc\equiv 2n(n+1)$}
\]
\begin{center}iff\end{center}
\[
\text{$\tilde{s}$ defines a toric Ricci-flat K\"ahler metric}
\]
and, when this happens, the corresponding toric Sasaki metric is 
Einstein.
\end{prop}
\begin{proof}
The relation between $Sc$ and $\widetilde{Sc}$ follows by direct
application of formula~(\ref{scalarsymp2}) for the scalar curvature
to the symplectic potentials $s$ and $\tilde{s}$.

The last statement follows from the above symplectic reduction
property of symplectic potentials and a well known fact in Sasaki
geometry (see~\cite{BG1}): on a Sasaki manifold of dimension $2n+1$ 
the following are equivalent:
\begin{itemize}
\item[(i)] the Sasaki metric is Einstein with scalar curvature equal
to $n(2n+1)$;
\item[(ii)] the transversal K\"ahler metric is Einstein with scalar
curvature equal to $2n(n+1)$;
\item[(iii)] the cone K\"ahler metric is Ricci-flat.
\end{itemize}
\end{proof}

\section{New Sasaki-Einstein from Old K\"ahler-Einstein}
\label{s:newold-2}

In 1982 Calabi~\cite{C2} constructed, in local complex coordinates, a general
$4$-parameter family of $U(n)$-invariant extremal K\"ahler metrics, which he
used to put an extremal K\"ahler metric on
\[
H^{2n}_m = \CP (\Oo(-m)\oplus\C) \to \C\CP^{n-1}\,,
\]
for all $n,m\in\N$ and any possible K\"ahler cohomology class. In particular,
when $n=2$, on all Hirzebruch surfaces.

When written in action-angle coordinates, using symplectic potentials, Calabi's 
family can be seen to contain many other interesting cohomogeneity one special
K\"ahler metrics. Besides the ones discussed in~\cite{Abr4} and some of the 
Bochner-K\"ahler orbifold examples presented in~\cite{Abr3}, it also contains a 
$1$-parameter family of K\"ahler-Einstein metrics that are directly 
related to the Sasaki-Einstein metrics constructed by Gauntlett-Martelli-Sparks-Waldram~\cite{GW1,GW2} in 2004.

Consider symplectic potentials $s_A: \breve{P_A} \subset (\R^+)^n \to \R$ of the 
form
\[
s_A(x) = \frac{1}{2}\left(\sum_{i=1}^n \left(x_i + \frac{1}{n+1}\right)
\log \left(x_i + \frac{1}{n+1}\right)  + h_A (r)\right)\,,
\]
where 
\[
r=x_1 + \cdots + x_n\,,
\] 
the polytope $\breve{P_A}$ will be determined below and
\[
h_A''(r) = - \frac{1}{r+\frac{n}{n+1}} + 
\frac{(r+\frac{n}{n+1})^{n-1}}{p_A (r)}\,,
\]
with
\begin{equation}\label{eq:p_A-1}
p_A (r) := \left(r+\frac{n}{n+1}\right)^n 
\left(\frac{1}{n+1}-r\right) - A
\quad\text{and}\quad
0 < A < \frac{n^n}{(n+1)^{n+1}}\,.
\end{equation}
One can check (see~\cite{Abr4}) that this family of symplectic potentials defines 
a $1$-parameter family of local K\"ahler-Einstein metrics with $Sc = 2n(n+1)$.

Let $-a$ and $b$ denote the first negative and positive zeros of $p_A$. Then
\begin{equation}\label{eq:p_A-2}
p_A(r) = (r+a)(b-r)q_A(r)\,,
\end{equation}
where $q_A$ is a polynomial of degree $n-1$,
\[
0<a<\frac{n}{n+1}\,,\quad 0<b<\frac{1}{n+1}\quad\text{and}
\]
\[ 
\left(\frac{n}{n+1}-a\right)^n \left(\frac{1}{n+1}+a\right) = A =
\left(\frac{n}{n+1}+b\right)^n \left(\frac{1}{n+1}-b\right)\,.
\]
From~(\ref{eq:p_A-1}) and~(\ref{eq:p_A-2}) we get that
\begin{align}
p_A'(r) & = - (n+1) r \left(r + \frac{n}{n+1}\right)^{n-1} \notag \\
& = (b-r) q_A (r) - (r+a) q_A (r) + (r+a)(b-r) q_A'(r)\,, \notag
\end{align}
which for $r=-a$ and $r=b$ implies that:
\begin{align}
q_A(-a) & = \frac{(n+1) a}{a+b} \left(\frac{n}{n+1}-a\right)^{n-1} \notag \\
q_A(b) & = \frac{(n+1) b}{a+b} \left(b+\frac{n}{n+1}\right)^{n-1}\,. \notag
\end{align}
This means in particular that
\begin{align}
\frac{(r+\frac{n}{n+1})^{n-1}}{p_A (r)} & =
\frac{(r+\frac{n}{n+1})^{n-1}}{(r+a)(b-r)q_A(r)} \notag \\
& = \frac{\frac{1}{(n+1) a}}{r+a} + \frac{\frac{1}{(n+1) b}}{b-r} +
\frac{\cdots}{q_A (r)}\,.\notag
\end{align}
Hence, the symplectic potential $s_A$ defines a K\"ahler-Einstein metric
with $Sc=2n(n+1)$ on the toric quasifold determined by the polytope 
$P_A\subset\R^n$ defined by the following inequalities:
\[
x_i +\frac{1}{n+1} \geq 0\,, \ i=1,\ldots,n\,,\quad
\frac{1}{(n+1) a} (r+a) \geq 0 \quad\text{and}\quad
\frac{1}{(n+1) b} (b-r) \geq 0\,.
\]

Since $P_A$ is never $GL(n,\R)$ equivalent to a Delzant polytope, these
K\"ahler-Einstein quasifolds do not give rise to any interesting  
K\"ahler-Einstein smooth manifolds. However, they do give rise to
interesting Sasaki-Einstein smooth manifolds. In fact, for suitable
values of the parameter $A$, the polytope $P_A$ determines 
via~(\ref{eq:stdcone}) a standard cone $C_A \subset \R^{n+1}$  that is
$GL(n+1,\R)$ equivalent to one of the good cones $C(k,m)\subset\R^{n+1}$
defined in the Introduction. The Boothby-Wang symplectic potential
\[
\tilde{s}_A : C_A \subset \R^{n+1} \to \R\,,
\]
determined by $s_A$ via~(\ref{eq:BW-pot}), will then define
a Ricci-flat K\"ahler metric on the toric ``quasicone"
determined by $C_A$ and, for these appropriate values of $A$,
also a Ricci-flat K\"ahler metric on the smooth toric symplectic
cone determined by the appropriate $C(k,m)$ and a Sasaki-Einstein
metric on the corresponding smooth toric contact manifold, thus
proving Theorem~\ref{thm:main}.

The facets of $C_A$ are defined by the following set of defining normals:
\begin{align}
\nu'_i & = \left(\ve_i, \frac{1}{n+1}\right)\,,\ i=1,\ldots,n\,;\notag\\
\nu'_a & = \left(\frac{\vd}{(n+1)a}, \frac{1}{n+1}\right)\,;\notag\\
\nu'_b & = \left(-\frac{\vd}{(n+1)b}, \frac{1}{n+1}\right)\,;\notag
\end{align}
where 
\[
\text{$\ve_i\in\R^n\,,\ i=1,\ldots,n$, are the canonical basis vectors and}
\ \vd = \sum_{i=1}^n \ve_i \in \R^n\,.
\]
To suitably express the condition implying that the cone $C_A$ is 
$GL(n+1,\R)$ equivalent to one of the good cones $C(k,m)\subset\R^{n+1}$,
it is convenient to introduce the following auxiliar real parameter:
\[
\la_A := \frac{b}{a}\cdot\frac{n-(n+1)a}{n+(n+1)b}\,.
\]
Note that, as $A$ varies in the open interval $(0, \frac{n^n}{(n+1)^{n+1}})$, 
$\la_A$ assumes all values in the open interval $(0,1)$. 

\begin{prop}
If $\la_A \in \left(0,1\right)$ can be written in the form
\begin{equation} \label{eq:lambda}
\la_A = \frac{kn-m}{n+m}\,,
\end{equation}
with $k,m\in\N$ satisfying
\begin{equation} \label{cond:kse-2}
\frac{(k-1)n}{2} < m < kn\,,
\end{equation}
then $C_A$ is $GL(n+1,\R)$ equivalent to the cone 
$C(k,m)\subset\R^{n+1}$ defined by the following normals:
\begin{align}
\nu_i & = \left(\ve_i, 1\right)\,,\ i=1, \ldots, n-1\,; \notag \\
\nu_{n} & = \left((m+1)\ve_n - \vd, 1\right)\,; \label{eq:normals} \\
\nu_- & = \left( k \ve_n, 1 \right)\,; \notag \\
\nu_+ & = \left(- \ve_n, 1\right)\,. \notag 
\end{align}
\end{prop}

\begin{proof}
Consider $T\in GL(n+1,\R)$ defined by
\begin{align}
T^t (\ve_i, 0) & = \left(\ve_i - \gamma \ve_n, 0\right)\,,\ i=1, \ldots, n-1\,, \notag \\
T^t (\ve_n, 0) & = \left((m+1-\gamma)\ve_n - \vd, 0\right)\,, \notag \\
T^t (\vo,1) & = \left((n+1)\gamma \ve_n, n+1\right)\,, \notag 
\end{align}
for some $\gamma\in\R$. Then:
\begin{align}
T^t(\nu'_i) & = \nu_i\,,\ i=1, \ldots, n\,; \notag \\
T^t(\nu'_a) & = \nu_- \quad\text{iff}\quad 
\gamma = \frac{k(n+1)a-m}{(n+1)a - n}\,; \notag \\
T^t(\nu'_b) & = \nu_+ \quad\text{iff}\quad 
\gamma = \frac{m-(n+1)b}{n + (n+1)b}\,. \notag 
\end{align}
This implies that $C_A$ is $GL(n+1,\R)$ equivalent to $C(k,m)$ provided
\[
\frac{k(n+1)a-m}{(n+1)a - n} = \frac{m-(n+1)b}{n + (n+1)b}\,,
\]
which is equivalent to~(\ref{eq:lambda}).
\end{proof}

\begin{rem} \label{rem:reeb}
Note that, in the action-angle coordinates associated with the cone
$C(k,m)$, the Reeb vector field of the Ricci-flat K\"ahler-Sasaki metric
is 
\[
K=({\bf 0}, T^t (\vo,1)) \quad\text{with}\quad 
T^t (\vo,1)) = \left((n+1)\gamma \ve_n, n+1\right))\,.
\] 
Since
\[
\gamma = \frac{k(n+1)a-m}{(n+1)a - n} = \frac{m-(n+1)b}{n + (n+1)b}\,,
\]
the (ir)regularity of $K$ is determined by the (ir)rationality of the
admissible values of $a$ or, equivalently, $b$.
\end{rem}

When $k=1$ we have $0<m<n$ and each cone $C(1,m)\subset\R^{n+1}$ is the 
standard cone over the integral Delzant polytope $P(m)\subset\R^n$ defined 
by the following affine functions:
\begin{align}
\ell_i (x) & = \langle x,\ve_i\rangle + 1\,,\ i=1, \ldots, n-1\,; \notag \\
\ell_{n} (x) & = \langle x, (m+1)\ve_n - \vd\rangle + 1\,; \notag \\
\ell_- (x) & = \langle x, \ve_n\rangle + 1\,; \notag \\
\ell_+ (x) & = \langle x, - \ve_n \rangle + 1\,. \notag 
\end{align}
If $n=2$ then $m=1$ and $P(1)\subset\R^2$ is well known to be a polytope for 
the first Hirzebruch surface:
\[
H^{4}_1 = \CP (\Oo(-1)\oplus\C) \to \C\CP^{1}\,,
\]
In fact, one easily checks that $P(m)\subset\R^n$, $0<m<n$, defines a 
smooth compact toric symplectic manifold $(H^{2n}_m, \omega)$ where 
\[
H^{2n}_m = \CP (\Oo(-m)\oplus\C) \to \C\CP^{n-1} \quad\text{and}\quad
[\omega] = 2\pi c_1 (H^{2n}_m)\,.
\]
Hence the Sasaki-Einstein manifold $N^{2n+1}_{1,m}$ is diffeomorphic to 
the corresponding Boothby-Wang manifold, cf. Example~\ref{ex:boothby-wang}, 
which is the circle bundle of the anti-canonical line bundle of 
$H^{2n}_m$. This proves Theorem~\ref{thm:main-2}.

\begin{rem} \label{rem:hirz}
In general, i.e. when $1<k\in\N$, the cones $C(k,m)\subset\R^{n+1}$
are standard cones over labeled polytopes $P(k,m)\subset\R^n$ and
the corresponding manifolds $N^{2n+1}_{k,m}$ are given by an
orbifold version of the Boothby-Wang construction.
\end{rem}

We will now check that, when $n=2$, the cones $C(k,m)\subset\R^3$,
with $k,m\in\N$ satisfying~(\ref{cond:kse-2}) and the simply connected
condition
\[
{\rm gcd} (m+n, k+1) = 1 \,,
\]
are $SL(3,\Z)$ equivalent to the cones $C_{p,q}\subset\R^3$ associated
to the Sasaki-Einstein $5$-manifolds $Y^{p,q}$, $0<q<p$, ${\rm gcd} (q, p) = 1$, 
constructed by Gauntlett-Martelli-Sparks-Waldram~\cite{GW1}. The defining
normals of the cones $C(k,m)\subset\R^3$ are $\nu_1$, $\nu_2$, $\nu_-$ and $\nu_+$ 
defined by~(\ref{eq:normals}) with $n=2$. According to~\cite{MSY}, the cones 
$C_{p,q}\subset\R^3$ have defining normals given by
\[
\mu_1 = (1, p-q-1, p-q)\,,\quad
\mu_2 = (1,1,0) \,,\quad 
\mu_- = (1,0,0) \quad\text{and}\quad
\mu_+ = (1,p,p)\,.
\]
Consider the linear map $T_{k,m}\in SL(3,\Z)$ defined
by the matrix
\[
\begin{bmatrix}
0 & 0 & 1 \\
k-m-1 & -1 & k \\
k-m & -1 & k
\end{bmatrix}\,.
\]
When $k=p-1$ and $m=p+q-2$ we have that
\[
T_{k,m} (\nu_1) = \mu_1 \,, \quad
T_{k,m} (\nu_2) = \mu_2 \,, \quad
T_{k,m} (\nu_-) = \mu_- \quad\text{and}\quad
T_{k,m} (\nu_+) = \mu_+\,,
\]
i.e. $T_{k,m}\in SL(3,\Z)$ provides the required equivalence.


\begin{thebibliography}{000}

\bibitem{Abr1} M.~Abreu, \textit{K\"{a}hler geometry of toric varieties and
extremal metrics}, Internat. J. Math. {\bf 9} (1998), 641--651. 

\bibitem{Abr2} M.~Abreu, \textit{K\"ahler geometry of toric manifolds in
symplectic coordinates}, in "Symplectic and Contact Topology: 
Interactions and Perspectives" (eds. Y.Eliashberg, B.Khesin and
F.Lalonde), Fields Institute Communications 35, 
American Mathematical Society, 2003, pp. 1--24.

\bibitem{Abr3} M.~Abreu, \textit{K\"ahler metrics on toric
    orbifolds}, J. Differential Geom. {\bf 58} (2001), 151--187.

\bibitem{Abr4} M.~Abreu, \textit{Toric K\"ahler Metrics: cohomogeneity one 
examples of constant scalar curvature in action-angle coordinates}, 
to appear in the proceedings of the Eleventh International Conference on 
Geometry, Integrability and Quantization, Varna, Bulgaria, June 5--10, 2009.

\bibitem{Bn} A.~Banyaga, \textit{The geometry surrounding the Arnold-Liouville 
theorem} in ''Advances in geometry´´ (eds. J.-L.Brylinski, R.Brylinski, 
V.Nistor, B.Tsygan and P.Xu), Progress in Mathematics 172,
Birkh\"auser, 1999.

\bibitem{BnM1} A.~Banyaga and P.~Molino, \textit{G\'eom\'etrie des formes de 
contact compl\`etement int\'egrables de type toriques}, in \textit{S\'eminaire 
Gaston Darboux de G\'eom\'etrie et Topologie Diff\'erentielle, 1991–1992 
(Montpellier)}, Univ. Montpellier II, Montpellier, 1993, pp. 1--25.

\bibitem{BnM2} A.~Banyaga and P.~Molino, \textit{Complete integrability in 
contact geometry}, Penn State preprint PM 197, 1996.

\bibitem{BW} W.M.~Boothby and H.C.~Wang, \textit{On contact manifolds}, 
Ann. of Math. (2) {\bf 68} (1958), 721–-734.

\bibitem{BG} C.P.~Boyer and K.~Galicki, \textit{A note on toric contact geometry}, 
J. of Geom. and Phys. {\bf 35} (2000), 288–-298.

\bibitem{BG1} C.P.~Boyer and K.~Galicki, \textit{Sasakian Geometry}, 
Oxford University Press, 2008.

\bibitem{BGL2} D.~Burns, V.~Guillemin, E.~Lerman, \textit{Kaehler metrics on singular toric varieties}, math.DG/0501311.

\bibitem{C2} E.~Calabi, \textit{Extremal K\"{a}hler metrics}, 
in ``Seminar on Differential Geometry'' (ed. S.T.Yau), 
Annals of Math. Studies 102, Princeton Univ. Press, 1982, 259--290.

\bibitem{CDG} D.~Calderbank, L.~David, P.~Gauduchon, 
\textit{The Guillemin formula and K\"ahler metrics on toric symplectic manifolds}, J. Symplectic Geom. {\bf 1} (2003), 767--784. 

\bibitem{De} T.~Delzant, \textit{Hamiltoniens p\'eriodiques et images
    convexes de l'application moment}, Bull. Soc. Math. France {\bf
    116} (1988), 315--339.

\bibitem{D1} S.~Donaldson, \textit{Remarks on gauge theory, complex 
geometry and $4$-manifold topology}, 
in ``Fields Medallists' Lectures'' (eds. M.F. Atiyah and
D. Iagolnitzer), World Scientific, 1997, 384--403.

\bibitem{D4} S.~Donaldson, \textit{Scalar curvature and stability
    of toric varieties}, J. Differential Geom. {\bf 62} (2002), 289--349.

\bibitem{GW1} J.~Gauntlett, D.~Martelli, J.~Sparks, D.~Waldram,
\textit{Sasaki-Einstein metrics on $S\sp 2\times S\sp 3$},
Adv. Theor. Math. Phys. {\bf 8} (2004), 711--734. 

\bibitem{GW2} J.~Gauntlett, D.~Martelli, J.~Sparks, D.~Waldram,
\textit{A new infinite class of Sasaki-Einstein manifolds},
Adv. Theor. Math. Phys. {\bf 8} (2004), 987--1000. 

\bibitem{Gui1} V.~Guillemin, \textit{K\"ahler structures on toric varieties},
J. Differential Geometry {\bf 40} (1994), 285--309.

\bibitem{GS} V.~Guillemin and S.~Sternberg, \textit{Geometric quantization and 
multiplicities of group representations}, Invent. Math. {\bf 67} (1982), 515--538.

\bibitem{Le} E.~Lerman, \textit{Contact toric manifolds}, 
J. Symplectic Geom. {\bf 1} (2003), 785--828.

\bibitem{Le1} E.~Lerman, \textit{Geodesic flows and contact toric manifolds},
in ''Symplectic Geometry of Integrable Hamiltonian Systems´´, Birkh\"auser, 
2003.

\bibitem{Le3} E.~Lerman, \textit{Maximal tori in the contactomorphism groups of 
circle bundles over Hirzebruch surfaces}, Math. Res. Lett. {\bf 10} (2003), 133--144.

\bibitem{Le2} E.~Lerman, \textit{Homotopy groups of $K$-contact toric manifolds}, 
Trans. Amer. Math. Soc. {\bf 356} (2004), 4075--4083.

\bibitem{LeTo} E.~Lerman and S.~Tolman, \textit{Hamiltonian torus actions
on symplectic orbifolds and toric varieties}, Trans. Amer. Math. Soc.
{\bf 349} (1997), 4201--4230.

\bibitem{MSY} D.~Martelli, J.~Sparks, S.-T.~Yau, 
\textit{The geometric dual of $a$-maximisation for toric Sasaki-Einstein 
manifolds}, Comm. Math. Phys. {\bf 268} (2006), 39--65. 

\bibitem{Pr} E.~Prato, \textit{Simple non-rational convex polytopes via symplectic geometry}, Topology {\bf 40} (2001), 961--975. 

\end{thebibliography}
\end{document}